\documentclass{article}
\usepackage[english]{babel}

\usepackage{theorem}
\usepackage{latexsym}
\usepackage{amsfonts}
\usepackage{amssymb}
\usepackage{multicol}

\newenvironment{proof}{\noindent{\it Proof. }}{\fbox{}\\}
\newtheorem{theorem}{Theorem}
\newtheorem{proposition}{Proposition}
\newtheorem{lemma}{Lemma}
\newtheorem{corollary}{Corollary}
{\theorembodyfont {\rmfamily}
\newtheorem{remark}{Remark}
\newtheorem{definition}{Definition}
\newtheorem{example}{Example}}

\begin{document}


\title{\bf{Stone-Weierstrass type theorems for large deviations}}

\author{Henri Comman\thanks{Department of Mathematics, University
of Santiago de Chile, Bernardo O'Higgins 3363, Santiago, Chile.
E-mail: hcomman@usach.cl}}

\date{}
\maketitle

\begin{abstract}
We give a general version of Bryc's theorem valid on any topological
space and with any algebra $\mathcal{A}$ of real-valued continuous
functions separating the points, or any well-separating class. In
absence of exponential tightness, and  when the underlying space is
locally compact regular and  $\mathcal{A}$ constituted by  functions
vanishing at infinity, we give a sufficient condition on the
functional $\Lambda(\cdot)_{\mid \mathcal{A}}$ to get large
deviations with not necessarily tight rate function. We obtain the
general variational form of any rate function on a completely
regular space; when either exponential tightness holds or  the space
is locally compact Hausdorff, we get it in terms of any algebra as
above.  Prohorov-type theorems are  generalized to any space, and
when it  is locally compact regular the exponential tightness
can be replaced by a (strictly weaker) condition on
$\Lambda(\cdot)_{\mid \mathcal{A}}$.
\end{abstract}

\bigskip

\hspace{10pt}\textbf{Key words}: Large deviations, rate function,
Bryc's theorem.

\hspace{10pt}2000 MSC. 60F10

\section{Introduction}

 Let $(\mu_{\alpha})$ be a net of
 Borel probability measures on a  topological space
 $X$ where compact sets are Borel sets, and let
  $(t_{\alpha})$  be a net in $]0,+\infty[$ converging to $0$. For
  any $[-\infty,+\infty[$-valued Borel measurable
  function $h$ on $X$, we write
$\mu_{\alpha}^{t_{\alpha}}(e^{h/t_{\alpha}})$ for $(\int_X
e^{h(x)/t_{\alpha}}\mu_{\alpha}(dx))^{t_{\alpha}}$, and define
$\overline\Lambda(h)=\log\limsup\mu_{\alpha}^{t_{\alpha}}(e^{h/t_{\alpha}})$,
$\underline\Lambda(h)=\log\liminf\mu_{\alpha}^{t_{\alpha}}(e^{h/t_{\alpha}})$,
and $\Lambda(h)=\log\lim\mu_{\alpha}^{t_{\alpha}}(e^{h/t_{\alpha}})$
when this limit exists (throughout this paper, "existence" of
$\Lambda(h)$ means existence in $[-\infty,+\infty]$). The Bryc's
theorem asserts that under exponential tightness hypothesis and when
$X$ is completely regular Hausdorff, the existence of
$\Lambda(\cdot)$ on the set $\mathcal{C}_b(X)$ of real-valued
bounded continuous functions on $X$, implies a large deviation
principle with rate function
$J(x)=\sup_{h\in\mathcal{C}_b(X)}\{h(x)-\Lambda(h)\}$ for all $x\in
X$ (\cite{bry}, \cite{dem}).

We prove here that this theorem  is still true replacing
$\mathcal{C}_b(X)$ by the bounded above part $\mathcal{A}_{ba}$ of
any algebra $\mathcal{A}$ of real-valued (not necessarily bounded)
continuous functions separating the points, or any well-separating
class; the rate function is then given  by
\begin{equation}\label{intro-eq1}
\forall x\in X,\ \ \ \ \
J(x)=\sup_{h\in\mathcal{A}_{ba}}\{h(x)-\Lambda(h)\}=
\sup_{h\in\mathcal{A}}\{h(x)-\underline{\Lambda}(h)\}=
\sup_{h\in\mathcal{A}}\{h(x)-\overline{\Lambda}(h)\},
\end{equation}
 and we can replace in the above expression $\mathcal{A}$ by its negative
 part $\mathcal{A}_-$
  when $\mathcal{A}$  does not vanish identically at any point of $X$
  (in particular when $\mathcal{A}$ contains the constants); this is the  most tricky part
   because of  the lack of stability by
   translation of $\mathcal{A}_-$.
 As a
consequence we obtain the variational form of any rate function
(under exponential tightness) on such a space  in terms of any such
a set $\mathcal{A}$; without assuming exponential tightness, we get
it only in terms of $\mathcal{C}(X)$ and $\mathcal{C}(X)_-$
(Corollary \ref{rate-funct-Polish}).

In fact, the main result  (Theorem \ref{Stone-W-exp-tight}) gives a
general version valid for any topological space $X$ with
 the  restriction that (\ref{intro-eq1}) holds only on  the
 "completely regular
 part" $X_0$ of $X$; however, (\ref{intro-eq1}) determines completely the rate function when
 $X$ is regular.
 Various kinds of
 hypotheses are proposed, which all reduce to exponential
   tightness in the completely regular (not necessarily Hausdorff)
   case; for instance, the simpler one requires  that  the tightening compact sets
   be
    included in
   $X_0$. This allows us to extend some known results  by relaxing
topological assumptions on the space: Corollary
\ref{general-Prohorov-LDP}  is a Prohorov-type result valid on any
topological space, when preceding versions assumed complete
regularity and Hausdorffness; Corollary \ref{appli-tvs} consider the
well-separating class constituted by continuous affine functionals
on any real Hausdorff topological vector space, when a preceding
version assumed metrizability.

When exponential tightness fails, but assuming $X$ locally compact
regular (not necessarily Hausdorff) and  $\mathcal{A}$ an algebra of
continuous functions vanishing at infinity which separates the
points and does not vanish identically at any point of $X$, we give
a necessary and sufficient condition on $\Lambda(\cdot)_{\mid
\mathcal{A}}$  to get large deviations with a rate function
satisfying a property weaker than the tightness, and having still
the form (\ref{intro-eq1}) (Theorem \ref{Stone-Wei-no-tight}). A
similar condition allows to get the large deviation principle  for
some subnets (Corollary \ref{local-compact-tight-free-Prohorov}).
When $X$ is moreover Hausdorff we show that a large deviation
principe is always governed by the rate function (\ref{intro-eq1})
with any $\mathcal{A}$ as above (Corollary
\ref{loc-compact-rate-funct}).

This is achieved by applying the  results of \cite{com-JTP-05} and
\cite{com-TAMS-03}; we use in particular the notion of approximating
class, that is a set of functions on   which the existence of
$\Lambda(\cdot)$  implies  large deviations lower bounds with a
function  satisfying (\ref{intro-eq3}). When $X$ is completely
regular, the upper bounds on compact sets hold also with the
function  (\ref{intro-eq3}), which turns to be the rate function
when large deviations hold; in the general case  some extra
conditions are required in order that (\ref{intro-eq3}) be a rate
function on $X_0$. Note that in absence of regularity the
identification of a rate function is quite difficult by the lack of
uniqueness.

More precisely,
 an approximating class is a set $\mathcal{T}$ of
$[-\infty,+\infty[$-valued continuous functions on $X$ such that
 for each $x\in X$,  each open set
$G$ containing $x$, each real $s>0$, and each real $t>0$, there
exists $h\in\mathcal{T}$ satisfying
\[
e^{-t}1_{\{x\}}\le e^{h}\le 1_G\vee e^{-s}.
\]
In \cite{com-JTP-05} we proved  that if $\Lambda(\cdot)$ exists on
$\mathcal{T}$ and under some extra condition (namely, $(iii)$ of
Proposition \ref{rate-funct}),  then
 $(\mu_{\alpha})$ satisfies a large deviation principle with powers
$(t_{\alpha})$ and  a rate function verifying
\begin{equation}\label{intro-eq3}
J(x)=\inf_{t>0}\sup_{\{h\in\mathcal{T}^{T}:h(x)\ge -t\}}
 \{-\Lambda(h)\}\ \ \ \ \ \ \ \textnormal{\ for all\ }x\in X_0,
 \end{equation}
 where $\mathcal{T}^T$ denotes the set of elements in
 $\mathcal{T}$ satisfying the usual tail condition of
 Varadhan's theorem. We first improve that by  showing that the $\sup$ in
 (\ref{intro-eq3})
 can
be taken on $\mathcal{T}$ and  $\mathcal{T}_-$ in place of
$\mathcal{T}^T$, and $\Lambda(h)$ can be replaced by
$\overline\Lambda(h)$ and  $\underline\Lambda(h)$ (Proposition
\ref{rate-funct}). This result is used in a crucial way in the
sequel. Indeed, we first  show that the existence of
$\Lambda(\cdot)$ on $\mathcal{A}_{ba}$ implies the existence of
$\Lambda(\cdot)$ on $\mathcal{C}(X)_-$; next, we get large
deviations with rate function satisfying (\ref{intro-eq1}) on $X_0$
by applying
 Proposition \ref{rate-funct}
with  the approximating class $\mathcal{C}(X)$. The case with
well-separating class is proved similarly.

The paper is organized as follows. Section
\ref{section-rate-function} gives  the general form of a rate
function in terms of an approximating class, strengthening a result
of \cite{com-JTP-05}. In Section \ref{section-main-result} we
establish  the general versions of Bryc's theorem; an example of
Hausdorff  regular  space where the usual form does not work only
because of one point, but satisfying our general hypotheses is
given. In Section \ref{section-locally-compact} we study the case
where $X$ is locally compact regular.

\section{General form of a rate function}\label{section-rate-function}

Throughout the paper $X$ denotes a topological space in which
compact sets are Borel sets (in particular, no separation axiom is
required) and  $\mathcal{C}(X)$ denotes the set of all real-valued
continuous functions on $X$. We exclude the trivial case where $X$
is reduced to one point. We recall that a set
$\mathcal{C}\subset\mathcal{C}(X)$ "separates
 the points of $X$" if for any pair of points $x\neq y$ in $X$ there exists
 $h\in\mathcal{C}$ such that $h(x)\neq h(y)$. By
 "$\mathcal{C}$ does not vanish identically at any point of $X$" we
 mean that for any $x\in X$ there exists $h\in\mathcal{C}$ such that $h(x)\neq
 0$. Note that any well-separating class
  satisfies the two above
 properties.
 Let $\mathcal{F}$, $\mathcal{G}$, $\mathcal{K}$ denote respectively
 the set of closed, open, and compact subsets of $X$, and let $l$ be
 a $[0,+\infty]$-valued function on $X$. We say that  $(\mu_{\alpha})$ satisfies   the large deviation upper
(resp. lower) bounds with  powers $(t_{\alpha})$ and  function $l$
if
\begin{equation}\label{defi-upper-ldp}
\limsup\mu_{\alpha}^{t_{\alpha}}(F)\le\sup_{x\in F}e^{-l(x)}\ \ \ \
\ \ \ \textnormal{for all  $F\in\mathcal{F}$},
\end{equation}
\begin{equation}\label{defi-lower-ldp}
(\textnormal{resp. \ \ \ \ \ } \sup_{x\in G}e^{-l(x)}\le
 \liminf\mu_{\alpha}^{t_{\alpha}}(G)\ \ \
\ \ \ \ \textnormal{for all  $G\in\mathcal{G}$}).
\end{equation}
When (\ref{defi-upper-ldp}) (resp. (\ref{defi-upper-ldp}) with
$\mathcal{K}$ in place of $\mathcal{F}$) and  (\ref{defi-lower-ldp})
hold, we say that $(\mu_{\alpha})$ satisfies a  large deviation
principle (resp. vague large deviation principle)   with powers
$(t_{\alpha})$; in this case, the lower-regularization of $l$ (i.e.,
the greatest lower semi-continuous function lesser than $l$) is
called a rate function, which is said to be tight when it has
compact level sets.  As it is well known (\cite{dem}, Lemma 4.1.4
and Remark pp. 118), when $X$ is regular a rate function is uniquely
determined and coincides with the function $l_0$ defined by
\[
l_0(x)= -\log\inf\{\limsup\mu_{\alpha}^{t_{\alpha}}(G):
G\in\mathcal{G}, G\ni x\}\ \ \ \ \ \textnormal{for all $x\in X$}.
\]

The following proposition will be used in the sequel since we will
deal with functions which do not necessarily satisfy the
 usual tail condition
\begin{equation}\label{tail}
 \lim_{M\rightarrow\infty}\limsup\mu_{\alpha}^{t_{\alpha}}
(e^{h/t_{\alpha}}1_{\{e^h\ge M\}})=0.
\end{equation}
It is easy to see that it generalizes Varadhan's theorem; indeed,
since for each $[-\infty,+\infty[$-valued Borel function $h$ on $X$,
 for each
 subnet $(\mu_{\beta}^{t_{\beta}})$ of
$(\mu_{\alpha}^{t_{\alpha}})$  and each real $M$ we have
\[\limsup
\mu_{\beta}^{t_{\beta}}(e^{h/t_{\beta}})=\limsup
\mu_{\beta}^{t_{\beta}}(e^{h/t_{\beta}}1_{\{h<M\}})\vee\limsup
\mu_{\beta}^{t_{\beta}}(e^{h/t_{\beta}}1_{\{h\ge M\}}),\]
 it
follows that when $h\in \mathcal{C}(X)$ and satisfies (\ref{tail}),
by letting $M\rightarrow+\infty$  Proposition
\ref{extension-Varadh} implies $\limsup
\mu_{\beta}^{t_{\beta}}(e^{h/t_{\beta}})=\sup_{x\in X}e^{h(x)-l(x)}$
hence the existence of $\Lambda(h)$ (since this expression  does not
depend on the subnet along which the upper limit is taken).

\begin{proposition}\label{extension-Varadh}
If  the large deviation upper (resp. lower) bounds  hold with some
function $l$, then for each $h\in \mathcal{C}(X)$ and each real $M$
we have
\[
\limsup\mu_{\alpha}^{t_{\alpha}}(e^{h/t_{\alpha}}1_{\{h\le
M\}})\le\sup_{x\in X,h(x)\le M}e^{h(x)-l(x)}
\]
\[
(resp.\  \ \ \ \liminf
\mu_{\alpha}^{t_{\alpha}}(e^{h/t_{\alpha}}1_{\{h< M\}})\ge
\sup_{x\in X, h(x)< M}e^{h(x)-l(x)}).
\]
In particular when the upper and the lower bounds hold with $l$ we
have
\[
\lim_{M\rightarrow+\infty}\liminf\mu_{\alpha}^{t_{\alpha}}(e^{h/t_{\alpha}}1_{\{h<
M\}})=\lim_{M\rightarrow+\infty}\limsup\mu_{\alpha}^{t_{\alpha}}(e^{h/t_{\alpha}}1_{\{h\le
M\}})=\sup_{x\in X}e^{h(x)-l(x)}.
\]
\end{proposition}

\begin{proof}
For each $[-\infty,+\infty[$-valued Borel measurable function $h$ on
$X$, each $\varepsilon>0$, and each $x\in X$,  we put
$G_{e^{h(x)},\varepsilon}=\{y\in X:
e^{h(x)}-\varepsilon<e^{h(y)}<e^{h(x)}+\varepsilon\}$ and
$F_{e^{h(x)},\varepsilon}=\{y\in X: e^{h(x)}-\varepsilon\le
e^{h(y)}\le e^{h(x)}+\varepsilon\}$. First assume that the large
deviation lower bounds  hold with $l$.
 For each real $M$, by Theorem
3.1 of \cite{com-TAMS-03} applied to the function
$k=h1_{\{h<M\}}-\infty 1_{\{h\ge M\}}$ (with the convention
"$\infty\cdot 0=0$"),  there exists a subnet
$(\mu_{\beta}^{t_{\beta}})$ of $(\mu_{\alpha}^{t_{\alpha}})$ such
that
\[\liminf\mu_{\alpha}^{t_{\alpha}}(e^{k/t_{\alpha}}1_{\{h<M\}})=
\limsup\mu_{\beta}^{t_{\beta}}(e^{k/t_{\beta}}1_{\{h<M\}})\]\[
\sup_{\{x\in X, \varepsilon>0: h(x)< M
\}}\{(e^{k(x)}-\varepsilon)\limsup\mu_{\beta}^{t_{\beta}}(
G_{e^{k(x)},\varepsilon}\cap\{h<M\})\}.\] Since
 $k$ coincides with $h$ on $\{h<M\}$ we have
 $G_{e^{k(x)},\varepsilon}\cap\{h<M\}=G_{e^{h(x)},\varepsilon}\cap\{h<M\}$,
hence
\[\liminf\mu_{\alpha}^{t_{\alpha}}(e^{h/t_{\alpha}}1_{\{h<M\}})=\limsup
\mu_{\beta}^{t_{\beta}}(e^{h/t_{\beta}}1_{\{h<M\}})=\]
\[\sup_{\{x\in X, \varepsilon>0: h(x)< M
\}}\{(e^{h(x)}-\varepsilon)\limsup\mu_{\beta}^{t_{\beta}}(
G_{e^{h(x)},\varepsilon}\cap\{h<M\})\}\ge\sup_{x\in X,
h(x)<M\}}e^{h(x)}e^{-l(x)}\]
 (where the last inequality follows from the large deviations lower
bounds), which proves the lower bounds case. Assume now that
 the large deviation upper bounds  hold with $l$, and
suppose that
\[\limsup\mu_{\alpha}^{t_{\alpha}}
(e^{h/t_{\alpha}}1_{\{e^h\le M\}})>\sup_{x\in X, h(x)\le
M}e^{h(x)-l(x)}\]for some real  $M$.  Applying Theorem 3.1 of
\cite{com-TAMS-03} as above  with $F_{e^{h(x)},\varepsilon}$  in
place of $G_{e^{h(x)},\varepsilon}$ yields
\[\limsup
\mu_{\alpha}^{t_{\alpha}}(e^{h/t_{\alpha}}1_{\{h\le
M\}})=\sup_{\{x\in X, \varepsilon>0: h(x)\le M
\}}\{(e^{h(x)}-\varepsilon)\limsup\mu_{\alpha}^{t_{\alpha}}(
F_{e^{h(x)},\varepsilon}\cap\{h\le M\})\},\] and therefore
 there  exists
$x\in X$  and $\varepsilon>0$ such that
\[
(e^{h(x)}-\varepsilon)\limsup\mu_{\alpha}^{t_{\alpha}}(
F_{e^{h(x)},\varepsilon}\cap\{h\le M\})>\sup_{x\in X, h(x)\le M}
e^{h(x)-l(x)}.
\]
By the  large deviation upper-bounds we have
\[
(e^{h(x)}-\varepsilon)\sup_{y\in F_{e^{h(x)},\varepsilon}\cap\{h\le
M\}} e^{-l(y)}>\sup_{x\in X, h(x)\le M} e^{h(x)-l(x)},
\]
and so there exists $x'\in F_{e^{h(x)},\varepsilon}\cap\{h\le M\}$
such that
\[
(e^{h(x)}-\varepsilon) e^{-l(x')}>\sup_{x\in X, h(x)\le M}
e^{h(x)}e^{-l(x)}.\]Since $e^{h(x')}\ge e^{h(x)}-\varepsilon$ we
obtain $e^{h(x')} e^{-l(x')}>\sup_{x\in X, h(x)\le M} e^{h(x)-l(x)}$
and the contradiction, which  proves the upper bounds case.
\end{proof}

We recall here the definition of an approximation class, which
involves the set $X_0$ constituted by  the points $x\in X$ which can
be suitably separated  by a continuous function from any closed set
not containing $x$. Note that ${\mathcal{C}(X)}_-$ is an
approximating class for any space $X$. It is known that $X_0=X$ if
and only if $X$ is completely regular (\cite{com-JTP-05},
Proposition 1). At the other extreme,  $X_0=\emptyset$ when
$\mathcal{C}(X)$ is reduced to constants and  $X$ is  a $T_0$
space containing more than one point, as it may occur with some
regular Hausdorff spaces (\cite{dou}). Note also that the negative
part $\mathcal{A}_-$ of any approximating class $\mathcal{A}$ is
again an approximating class.

\begin{definition}\label{defi-approx-class}
Let $X_0$ be the set of points $x$ of $X$ such that for each
$G\in\mathcal{G}$ containing $x$, each real $s>0$, and each real
$t>0$, there exists $h\in\mathcal{C}_b(X)$
 such that
\begin{equation}\label{defi-approx-class-eq0}
e^{-t}1_{\{x\}}\le e^{h}\le 1_G\vee e^{-s}.
\end{equation}
A class $\mathcal{T}$ of $[-\infty,+\infty[$-valued continuous
functions on $X$
 is said to be  approximating if for each $x\in X_0$,  each
$G\in\mathcal{G}$ containing $x$, each real $s>0$, and each real
$t>0$, $\mathcal{T}$ contains some function  satisfying
(\ref{defi-approx-class-eq0}).
\end{definition}

We introduce now a strong  variant of  exponential tightness by
requiring that the tightening  compact sets  be included in $X_0$.
Of course, it coincides with the usual one in the completely regular
  case.

\begin{definition}\label{extension-exp-tightness}
 The net $(\mu_\alpha)$ is
$X_0$-exponentially tight with respect to $(t_\alpha)$  if for each
$\varepsilon>0$ there exists a compact set $K\subset X_0$ such that
$\limsup\mu_{\alpha}^{t_{\alpha}}(X\verb'\'K)<\varepsilon$.
\end{definition}

To any approximating class $\mathcal{T}$ we  associate the function
$l_{\mathcal{T}}$ defined by

\begin{equation}\label{rate-funct-eq0}
l_{\mathcal{T}}(x)=\left\{
\begin{array}{ll}
\inf_{t>0}\sup_{\{h\in\mathcal{T}^{T}:h(x)\ge -t\}}
 \{-\overline{\Lambda}(h)\} & if\  x\in X_0
 \\ \\
\sup_{G\supset\{x\},G\in\mathcal{G}}\sup_{0<s<\infty}
\inf_{\mathcal{A}_{G,s}}
 \{-\overline{\Lambda}(h)\} & if\  x\in X\verb'\' X_0,
 \end{array}
 \right.
 \end{equation}
 where $\mathcal{T}^T$ denotes the elements $h\in\mathcal{T}$
 satisfying the tail condition (\ref{tail}), and
\[
\mathcal{A}_{G,s}=\bigcup_{x\in G\cap X_0
,t>0}\{h\in\mathcal{A}:e^{h}\le 1_{G}\vee e^{-s}, h(x)\ge -t\}\ \ \
\textnormal{\ for all\ }G\in\mathcal{G}\ \textnormal{and}\
s\in]0,+\infty[.
 \]

In Theorem 3 of \cite{com-JTP-05} we proved that  the existence of
$\Lambda(\cdot)$ on $\mathcal{T}$   together with the condition
 $(iii)$ below  imply a large
 deviation principle  with rate function $l_{\mathcal{T}}$; in fact,
it is easy to verify that  the existence of $\Lambda(\cdot)$ on
$\mathcal{T}_-$ together with
 $(iii)$ are   sufficient.
 The following proposition  shows that we can replace
 $\mathcal{T}^{T}$ by $\mathcal{T}$ (resp. $\mathcal{T}_-$)  and
 $\overline{\Lambda}(h)$ by $\underline{\Lambda}(h)$ in
(\ref{rate-funct-eq0})
 for the case
  $x\in X_0$.
We can even replace $\overline{\Lambda}(h)$ and
$\underline{\Lambda}(h)$ by $\lim_M \underline{\Lambda}(h_M)$, where
$h_M=h 1_{\{h<M\}}-\infty 1_{\{h\ge M\}}$ for all $h\in\mathcal{T}$
and all reals $M$.
 This  will be  used in the next section
  in order to obtain the expression of the
 rate function in terms of the whole algebra (or the well-separating
 class)
   since this one may contain
 unbounded functions.

\begin{proposition}\label{rate-funct}
Consider
 the following
statements:
 \begin{itemize}
 \item[(i)] $(\mu_\alpha)$ is
$X_0$-exponentially tight with respect to $(t_\alpha)$;
\item[(ii)] $(\mu_\alpha)$ is
exponentially tight with respect to $(t_\alpha)$ and ${l_0}_{\mid
X\verb'\'X_0}=+\infty$;
 \item[(iii)] For all $F\in\mathcal{F}$,
  for all open covers $\{G_i:i\in I\}$ of $F\cap X_0$ and for all $\varepsilon>0$,
  there exists a finite set $\{G_{i_1},...,G_{i_N}\}\subset\{G_i:i\in I\}$ such that
 \[
\limsup\mu_{\alpha}^{t_{\alpha}}({F})-\limsup\mu_{\alpha}^{t_{\alpha}}(\bigcup_{1\le
j\le N}\overline{G_{i_j}})<\varepsilon.
\]
 \end{itemize}The following conclusions hold.
 \begin{itemize}
 \item[(a)] $(i)\Rightarrow(iii)$,  $(ii)\Rightarrow(iii)$, and $(i)\Rightarrow(ii)$ when
 $\mathcal{K}\subset\mathcal{F}$.
 \item[(b)]  If $(iii)$ holds and $\Lambda(\cdot)$ exists on the
 negative part $\mathcal{T}_-$ of an approximating class
 $\mathcal{T}$,
   then $(\mu_{\alpha})$ satisfies a large deviation principle with
powers $(t_{\alpha})$ and rate function $l_{\mathcal{T}}$. Moreover,
$l_{\mathcal{T}}=l_{\mathcal{T}_-}$ and
\begin{equation}\label{rate-funct-eq1}
\forall x\in X_0,\ \ \ \ \ l_{\mathcal{T}}(x)=
\inf_{t>0}\sup_{\{h\in\mathcal{T}_-:h(x)\ge -t\}}
 \{-{\Lambda}(h)\}=\inf_{t>0}\sup_{\{h\in\mathcal{T}:h(x)\ge -t\}}
 \{\lim_{M\rightarrow+\infty}-\underline{\Lambda}(h_M)\}.
 \end{equation}
   If only one  rate function can govern the large deviations, then
\[
\forall x\in X,\ \ \ \ \
l_{\mathcal{T}}(x)=\sup_{G\in\mathcal{G},G\ni x}\inf_{y\in G\cap
X_0}l_{\mathcal{T}}(y).
 \]
 \end{itemize}
 \end{proposition}

\begin{proof}
$(i)\Rightarrow(ii)$ is clear when
 $\mathcal{K}\subset\mathcal{F}$, since in this case we have for
 each
 compact $K\subset X_0$,
 \[\limsup\mu_{\alpha}^{t_{\alpha}}(X\verb'\'K)\ge
 \sup_{X\verb'\'K}e^{-l_0}\ge\sup_{X\verb'\'X_0}e^{-l_0}.\]
 Assume that $(ii)$ holds.
  Let $F\in\mathcal{F}$, let $\{G_i:i\in I\}$ be an open cover of
$F\cap X_0$ and let $\varepsilon>0$. By hypotheses, each $x\in
X\verb'\'X_0$ has an  open neighborhood $G_x$ such that
$\limsup\mu_\alpha^{t_\alpha}(G_x)<\varepsilon$.
 Let $K\in\mathcal{K}$
 satisfying
$\limsup\mu_\alpha^{t_\alpha}(X\verb'\'K)<\varepsilon$. Then
$\{G_i:i\in I\}\cup\{G_x:x\in X\verb'\'X_0\}$ is an open cover of
$F$ and a fortiori of $F\cap K$, hence there exists a finite set
$\{G_{i_1},...,G_{i_N},G_{x_1},...G_{x_M}\}$ covering $F\cap K$ so
that
\[\limsup\mu_{\alpha}^{t_{\alpha}}(F)=\limsup\mu_{\alpha}^{t_{\alpha}}(F\cap K)
\vee\limsup\mu_{\alpha}^{t_{\alpha}}(F\cap
X\verb'\'K)\le\limsup\mu_{\alpha}^{t_{\alpha}}(\bigcup_{1\le j\le
N}{G_{i_j}})+\varepsilon,
\]and $(iii)$ holds. If $(i)$ holds, then  the  finite cover can be obtained
from $\{G_i:i\in I\}$
 and the above expression is still valid.
This proves (a).

Assume that $(iii)$ holds and $\Lambda(\cdot)$ exists on the
negative part
 $\mathcal{T}_-$ of some
approximating class $\mathcal{T}$. By Theorem 3 of \cite{com-JTP-05}
(and the comment before Proposition \ref{rate-funct}),
$(\mu_{\alpha})$ satisfies a large deviation principle with powers
$(t_{\alpha})$ and rate function $l_{\mathcal{T}}$ defined in
(\ref{rate-funct-eq0}) with moreover
\[-\inf_{x\in G}l_{\mathcal{T}}(x)=
\inf_{0<s<\infty} \sup_{h\in\mathcal{T}_{G,s}}
 \Lambda(h)\ \ \ \ \ \textnormal{for all $G\in\mathcal{G}$}.\]
 By lower semi-continuity the above expression determines
 $l_{\mathcal{T}}$  and since
 $\mathcal{T}_{G,s}=(\mathcal{T}_-)_{G,s}$, the same reasoning with
 the approximating class
 $\mathcal{T}_-$  yields
 $l_{\mathcal{T}}=l_{\mathcal{T}_-}$. Let $x\in X_0$ and  $\nu$ be a real such that
$\nu>\sup_{t>0}\inf_{\{h\in\mathcal{T}:h(x)\ge
-t\}}\lim_M\underline{\Lambda}(h_M)$. For each $t>0$, there exists
$h_t\in\mathcal{T}$ such that $h_t(x)\ge-t$ and
$\nu>\underline{\Lambda}(h_t)$.  By Proposition
\ref{extension-Varadh} we get
\[\nu\ge\sup_{y\in X}\{h_t(y)-l_{\mathcal{T}}(y)\}\ge h_t(x)-l_{\mathcal{T}}(x)\ge -t
-l_{\mathcal{T}}(x),\]and finally   $\nu\ge-l_{\mathcal{T}}(x)$ by
letting $t\rightarrow 0$. Since $\nu$ is arbitrary, it follows that
\[-l_{\mathcal{T}}(x)\le\sup_{t>0}\inf_{\{h\in\mathcal{T}:h(x)\ge
-t\}}\lim_M\underline{\Lambda}(h_M),\]and   since
\[-l_{\mathcal{T}}(x)=-l_{\mathcal{T}_-}(x)=\sup_{t>0}\inf_{\{h\in\mathcal{T}_-:h(x)\ge
-t\}}\overline{\Lambda}(h),\] the two first assertions of (b)  are
proved. If only one rate function can govern the large deviations,
then the lower regularization
 of the function $l$ defined in Remark
\ref{bounds-with-no-rate-function}
 coincides with $l_{\mathcal{T}}$;
 consequently we have for all $x\in X$,
\[l_{\mathcal{T}}(x)=\sup_{G\in\mathcal{G},G\ni x}\inf_{y\in G} l(y)=
\sup_{G\in\mathcal{G},G\ni x}\inf_{y\in G\cap X_0}
l(y)=\sup_{G\in\mathcal{G},G\ni x}\inf_{y\in G\cap X_0}
l_{\mathcal{T}}(y).\]
\end{proof}

\begin{remark}\label{bounds-with-no-rate-function}
It is
 easy to see that $l_{\mathcal{T}}(x)=+\infty$\   for all $x\in
 X\verb'\'\overline{X_0}$\  (\cite{com-JTP-05}, Remark 1). Condition $(iii)$ in Proposition
 \ref{rate-funct} implies that
$\lim\mu_{\alpha}^{t_{\alpha}}(F)=0$ for all closed sets $F\subset
X\verb'\'X_0$. In fact, the proof of
 Proposition 3 of \cite{com-JTP-05} shows that
  the large deviation upper (resp.
lower) bounds hold also with the function $l$ defined by
\[
l(x)=\left\{
\begin{array}{ll}
l_{\mathcal{T}}(x) & \textnormal{if}\ x\in X_0
\\ \\
+\infty & \textnormal{if}\ x\in X\verb'\'X_0;
\end{array}
\right.
\]
 it follows that under exponential tightness $(ii)$ is equivalent to
$(iii)$ when $l$ is lower semi-continuous and $X$ regular.
\end{remark}

\section{Main result}\label{section-main-result}

In this section we establish  our general version  of Bryc's
theorem, whose usual algebraic statement in the completely regular
Hausdorff case is recovered by taking $\mathcal{A}=\mathcal{C}_b(X)$
in  Theorem \ref{Stone-W-exp-tight}; recall that in this case
$X_0=X$, so that the general hypotheses reduce to exponential
tightness, and (\ref{Stone-W-exp-tight-eq1}) coincides with
(\ref{intro-eq1}). The improvement is threefold: first,  it allows a
general separating algebra (resp. well-separating class)
$\mathcal{A}$; secondly, the rate function is obtained in terms of
the negative part $\mathcal{A}_-$ of $\mathcal{A}$ when
$\mathcal{A}$ does not vanish identically at any point of $X$;
finally, the results hold  for any topological space, under the
stronger hypothesis of $X_0$-exponential tightness  (or exponential
tightness plus
 some  extra conditions), and with the restriction that
the usual form of the rate function is obtained only on $X_0$.

 Let
us stress that more than the large deviation property itself, the
hard part consists  in obtaining the rate function in terms of
$\mathcal{A}$ and $\mathcal{A}_-$, respectively. To our knowledge,
up to now, in the algebraic case such formulas were known only when
$\mathcal{A}=\mathcal{C}_b(X)$ and $X$ completely regular Hausdorff.
 The proof is
heavily based on Proposition \ref{rate-funct}. For instance, it is
thanks to the expression (\ref{rate-funct-eq1}) of the rate function
together with Lemma \ref{lemma-continuity-compact-open} that we can
write
 (\ref{Stone-W-exp-tight-eq4}) and
(\ref{Stone-W-exp-tight-eq5})  leading to
 (\ref{Stone-W-exp-tight-eq1}).

\begin{lemma}\label{lemma-continuity-compact-open}
For each set $\mathcal{T}\subset\mathcal{C}(X)$ and each $x\in X$
  we have \[\sup_{t>0}\inf_{\{h\in\mathcal{T}:h(x)\ge -t\}}
\overline{\Lambda}(h)\ge\inf_{h\in\mathcal{T}}\{\overline{\Lambda}(h)-h(x)\}\]
and  \[\sup_{t>0}\inf_{\{h\in\mathcal{T}:h(x)\ge -t\}}
\lim_{M\rightarrow+\infty}\overline{\Lambda}(h_M)
\ge\inf_{h\in\mathcal{T}}\{\lim_{M\rightarrow+\infty}\overline{\Lambda}(h_M)-h(x)\}\]
with  equalities when $\mathcal{T}$ is stable by translations.
 The same holds  with $\underline{\Lambda}$ in place of
  $\overline{\Lambda}$.
\end{lemma}

\begin{proof}
 Let $\delta$ be a real such that
\[\sup_{t>0}\inf_{\{h\in\mathcal{T}:h(x)\ge
-t\}}\overline{\Lambda}(h)<\delta.\] For each $t>0$ there exists
$h_t\in\mathcal{T}$ such that $\overline{\Lambda}(h_t)<\delta$ and
$h_t(x)\ge-t$, hence $\overline{\Lambda}(h_t)-h_t(x)<\delta+t$ and
\[\inf_{h\in\mathcal{T}}\{\overline{\Lambda}(h)-h(x)\}\le
\inf_{t>0}\{\overline{\Lambda}(h_t)-h_t(x)\}\le\delta,\] which
proves the first inequality; the proof of the second one is similar.
The assertion about the equality is clear, as well as the last
assertion.
\end{proof}

\begin{theorem}\label{Stone-W-exp-tight}
Assume that $(\mu_\alpha)$ is exponentially tight with respect to
$(t_\alpha)$, and satisfies one of the conditions  of Proposition
\ref{rate-funct} (in particular under $X_0$-exponential tightness).
If $\Lambda(\cdot)$ exists on the bounded-above part
$\mathcal{A}_{ba}$ of a set $\mathcal{A}$ of real-valued continuous
functions on $X$, which is either an algebra separating the points
or a well-separating class, then $(\mu_\alpha)$ satisfies a large
deviation principle with powers $(t_\alpha)$ and rate function $J$
verifying $J_{\mid X\verb'\'\overline{X_0}}=+\infty$ and
\begin{equation}\label{Stone-W-exp-tight-eq1}
\forall x\in X_0,\ \ \ \ J(x)=
\sup_{h\in\mathcal{A}_{ba}}\{h(x)-\Lambda(h)\}=\sup_{h\in\mathcal{A}}\{h(x)-
\overline{\Lambda}(h)\}=\sup_{h\in\mathcal{A}}\{h(x)-
\underline{\Lambda}(h)\}\end{equation}
\[=\sup_{h\in\mathcal{A}}\{h(x)-\lim_{M\rightarrow+\infty}
\underline{\Lambda}(h_M)\},\] where  $h_M=h 1_{\{h<M\}}-\infty
1_{\{h\ge M\}}$ for all $h\in\mathcal{A}$ and all reals $M$. When
$X$ is regular (\ref{Stone-W-exp-tight-eq1}) determines uniquely the
rate function by
\[
\forall x\in X,\ \ \ \ J(x)=\sup_{G\in\mathcal{G},G\ni x}\inf_{y\in
G\cap X_0}J(y).
\]
When
$\mathcal{A}$ does not vanish identically at any point of $X$ (in
particular when $\mathcal{A}$ contains the constants as in the
well-separating case) it is sufficient to assume the existence of
$\Lambda(\cdot)$ on the negative part $\mathcal{A}_-$ of
$\mathcal{A}$ and  we can replace $\mathcal{A}_{ba}$ by
$\mathcal{A}_-$ in (\ref{Stone-W-exp-tight-eq1}).
\end{theorem}

\begin{proof}
Let $h\in\mathcal{C}(X)_-$ such that $\overline{\Lambda}(h)>-s$ for
some $s>0$,
 and put $h_s=h\vee-s$.
First assume that $\mathcal{A}$ is an algebra separating the points
  and put $g=\sqrt{-h_s}$.
 Let $\mathcal{B}$ be the
algebra generated by  $\mathcal{A}\cup\{c\}$ where $c$ is any
nonzero constant function, and note that any element
$g\in\mathcal{B}$ has the form $g=k+t$ for some $k\in\mathcal{A}$
and some constant $t$ (i.e., $\mathcal{B}=\mathcal{A}+\mathbb{R}$).
By the Stone-Weierstrass theorem, there is a net $(g_i)_{i\in I}$ in
$\mathcal{B}$ converging uniformly on compact sets to $g$. Put
$h_i=-g_i^2$, $k_i=h_i\vee -2s$ for all $i\in I$, and note that
$(h_i)$ and $(k_i)$ converge uniformly on compact sets to $h_s$. Let
$K\in\mathcal{K}$   such that
$\limsup\mu_{\alpha}^{t_{\alpha}}(X\verb'\'K)<e^{-3s}$. Assume that
$\Lambda(\cdot)$ exists on $\mathcal{A}_{ba}$, and note that this
gives the existence of $\Lambda(\cdot)$ on $\mathcal{B}_-$. Since
 for each
$i\in I$, and each subnet $(\mu_{\beta}^{t_{\beta}})$ of
$(\mu_{\alpha}^{t_{\alpha}})$ we have
\[e^{-2s}\le e^{\Lambda(k_i)}=\lim
\mu_{\beta}^{t_{\beta}}(e^{k_i/t_{\beta}})=\limsup
\mu_{\beta}^{t_{\beta}}(e^{k_i/t_{\beta}}1_K)\vee\limsup
\mu_{\beta}^{t_{\beta}}(e^{k_i/t_{\beta}}1_{X\verb'\'K}),\]and
\[e^{-s}\le\limsup
\mu_{\alpha}^{t_{\alpha}}(e^{h_s/t_{\alpha}})=\limsup
\mu_{\alpha}^{t_{\alpha}}(e^{h_s/t_{\alpha}}1_K)\vee\limsup
\mu_{\alpha}^{t_{\alpha}}(e^{h_s/t_{\alpha}}1_{X\verb'\'K}),\] it
follows that
$\lim\mu_{\alpha}^{t_{\alpha}}(e^{k_i/t_{\alpha}}{1_K})$ exists with
\begin{equation}\label{Stone-W-exp-tight-eq2}
e^{\Lambda(k_i)}=\lim\mu_{\alpha}^{t_{\alpha}}(e^{k_i/t_{\alpha}}{1_K})
\ \ \ \ \ \ \textnormal{for all $i\in I$},
\end{equation}
 and
\begin{equation}\label{Stone-W-exp-tight-eq2.1}
\limsup\mu_{\alpha}^{t_{\alpha}}(e^{h_s/t_{\alpha}})=
\limsup\mu_{\alpha}^{t_{\alpha}}(e^{h_s/t_{\alpha}}{1_K}).
\end{equation}
 The
inequalities
\[\log\mu_{\alpha}^{t_{\alpha}}(e^{k_i/t_{\alpha}}{1_K})-\sup_{x\in
K}|k_i(x)-h(x)|\le
\log\mu_{\alpha}^{t_{\alpha}}(e^{h_s/t_{\alpha}}{1_K})\le\]
\[\log\mu_{\alpha}^{t_{\alpha}}(e^{k_i/t_{\alpha}}{1_K})+\sup_{x\in
K}|k_i(x)-h_s(x)|\]combined with (\ref{Stone-W-exp-tight-eq2}) and
(\ref{Stone-W-exp-tight-eq2.1}) yield
\[\Lambda(k_i)-\sup_{x\in
K}|k_i(x)-h_s(x)|\le\log\liminf\mu_{\alpha}^{t_{\alpha}}(e^{h_s/t_{\alpha}}{1_K})\le
\log\liminf\mu_{\alpha}^{t_{\alpha}}(e^{h_s/t_{\alpha}})\le\]
\[\log\limsup\mu_{\alpha}^{t_{\alpha}}(e^{h_s/t_{\alpha}})
=\log\limsup\mu_{\alpha}^{t_{\alpha}}(e^{h_s/t_{\alpha}}{1_K})\le
\Lambda(k_i)+\sup_{x\in K}|k_i(x)-h_s(x)|,\] and by taking the limit
along  $i$,
 it follows that $\Lambda(h)$ exists with
\begin{equation}\label{Stone-W-exp-tight-eq3}
\Lambda(h)=\Lambda(h_s)=\lim\Lambda(k_i)=\lim\Lambda(h_i).
\end{equation}
 Since $h$ is arbitrary in
  ${\mathcal{C}(X)}_-$, $\Lambda(\cdot)$ exists on
${\mathcal{C}(X)}_-$ and
\begin{equation}\label{Stone-W-exp-tight-eq3.2}
\inf_{h\in{\mathcal{C}(X)}_-}\{\Lambda(h)-h(x)\}=
\inf_{h\in\mathcal{B}_-}\{\Lambda(h)-h(x)\}\ \ \ \ \ \textnormal{for
all $x\in X$}.
\end{equation}
 By
Proposition
 \ref{rate-funct}  and Lemma
\ref{lemma-continuity-compact-open}  with
$\mathcal{T}=\mathcal{C}(X)$ it follows that $(\mu_\alpha)$
satisfies a large deviation principle with powers $(t_\alpha)$ and
rate function $l_{\mathcal{C}(X)}$ taking infinite value outside
$\overline{X_0}$ (see Remark \ref{bounds-with-no-rate-function}) and
satisfying
\begin{equation}\label{Stone-W-exp-tight-eq4}
\forall x\in X_0,\ \ \ \ \ -l_{\mathcal{C}(X)}(x)=
\inf_{h\in\mathcal{C}(X)}
\{\lim_{M\rightarrow+\infty}\underline{\Lambda}(h_M)-h(x)\}=
\inf_{h\in\mathcal{C}(X)} \{\underline{\Lambda}(h)-h(x)\}\]
\[=\inf_{h\in{\mathcal{C}(X)}_-} \{{\Lambda}(h)-h(x)\}
=\inf_{h\in\mathcal{B}_-} \{{\Lambda}(h)-h(x)\}
\ge\inf_{h\in\mathcal{B}_{ba}}\{{\Lambda}(h)-h(x)\}=
\inf_{h\in\mathcal{A}_{ba}}\{{\Lambda}(h)-h(x)\}\]
\[\ge\inf_{h\in\mathcal{A}}\{\underline{\Lambda}(h)-h(x)\}
=\inf_{h\in\mathcal{B}}\{\underline{\Lambda}(h)-h(x)\}\ge\inf_{h\in\mathcal{C}(X)}
\{\underline{\Lambda}(h)-h(x)\},
\end{equation}
 where the fourth
equality follows from (\ref{Stone-W-exp-tight-eq3.2}), and
  the last two
  equalities follow
 by noting that $\overline{\Lambda}(h)=\overline{\Lambda}(k)+t$ when
$h=k+t$  with $k\in\mathcal{A}$ and $t\in\mathbb{R}$. Consequently,
  the above inequalities are equalities, (\ref{Stone-W-exp-tight-eq1})
holds and the first assertion of the algebraic case  is proved. The
assertion concerning the regular case follows from Proposition
\ref{rate-funct}.
 The last assertion follows by noting that when
$\mathcal{A}$ does not vanish identically at  any point of $X$ (in
particular when $\mathcal{A}$ contains the constants) we can use
$\mathcal{A}$ in place of
 $\mathcal{B}$,  and
it is sufficient to  assume the existence  of $\Lambda(\cdot)$ on
$\mathcal{A}_-$ .

Assume now that $\mathcal{A}$ is a well-separating class. Let
$\mathcal{A}_-^{\vee}$   be the set constituted by the finite maxima
of elements of $\mathcal{A}_-$.
 By Lemma 4.4.9 of \cite{dem} (which remains true for
 any
 topological space), for
each $K\in\mathcal{K}$ and each $\varepsilon>0$ there exists
$h_{K,\varepsilon}\in\mathcal{A}_-^{\vee}$
 such that
 $\sup_{x\in
K}|h_{K,\varepsilon}(x)-h(x)|<\varepsilon$.
 The nets $(h_i)_{i\in I}$ and $(k_i)_{i\in I}$,
where $k_i=h_i\vee -2s$ and
 $I=\{(K,\varepsilon):K\in\mathcal{K},\varepsilon>0\}$ (as a product directed
set), converge uniformly on compact sets to $h$. A similar proof as
above with $\mathcal{A}_-^{\vee}$ in place of $\mathcal{B}_-$ gives
 the existence of
$\Lambda(\cdot)$ on  ${\mathcal{C}(X)}_-$, and  the large deviation
principle with rate function $l_{\mathcal{C}(X)}$ satisfying
\begin{equation}\label{Stone-W-exp-tight-eq5}
\forall x\in X_0,\ \ \ \ \ -l_{\mathcal{C}(X)}(x)=
\inf_{h\in\mathcal{C}(X)}
\{\lim_{M\rightarrow+\infty}\underline{\Lambda}(h_M)-h(x)\}=
\inf_{h\in\mathcal{C}(X)}\{\underline{\Lambda}(h)-h(x)\}\]
\[
=\inf_{h\in{\mathcal{C}(X)}_-}\{\underline{\Lambda}(h)-h(x)\}=
\inf_{h\in\mathcal{A}_-^{\vee}}\{{\Lambda}(h)-h(x)\}
=\inf_{h\in\mathcal{A}_-}\{\Lambda(h)-h(x)\}\]
\[ \ge \inf_{h\in\mathcal{C}(X)}\{\underline{\Lambda}(h)-h(x)\},
\end{equation}
where the last equality follows by noting that
$\Lambda(h)=\max_{1\le j\le r}\Lambda(h_j)$ when $h=\bigvee_{j=1}^r
h_j$ with $\{h_j:1\le j\le r\}\subset\mathcal{A}$. Therefore, the
above inequality is an equality, which proves the well-separating
case.
\end{proof}

The following corollary gives the general variational form of any
rate function on  a completely regular space in terms of
$\mathcal{C}(X)$; this result seems new in absence of exponential
tightness. When $X$ is Polish or locally compact Hausdorff, it gives
the general form of any tight rate function in terms of any
separating algebra or well-separating class,  since in this case the
exponential tightness holds. In the locally compact Hausdorff case,
we may compare it with a similar result obtained in the next section
when the exponential tightness fails (Corollary
\ref{loc-compact-rate-funct}).

\begin{corollary}\label{rate-funct-Polish}
Let $X$ be completely regular. If  $(\mu_\alpha)$ satisfies a large
deviation principle with powers $(t_\alpha)$, then the  rate
function $J$ is given by
\[
J(x)=
\sup_{h\in\mathcal{C}(X)_-}\{h(x)-\Lambda(h)\}=\sup_{h\in\mathcal{C}(X)}\{h(x)-
\overline{\Lambda}(h)\}=\sup_{h\in\mathcal{C}(X)}\{h(x)-
\underline{\Lambda}(h)\}\]
\[=\sup_{h\in\mathcal{C}(X)}\{h(x)-\lim_{M\rightarrow+\infty}
\underline{\Lambda}(h_M)\}\ \ \ \ \ \ \textit{for all $x\in X$}.
\]
If moreover  $(\mu_\alpha)$ is exponentially tight with respect to
$(t_\alpha)$, then we can replace $\mathcal{C}(X)$
 by $\mathcal{A}$  in the last two above
equalities, and $\mathcal{C}(X)_-$ by $\mathcal{A}_{ba}$ (resp.
$\mathcal{A}_-$ when $\mathcal{A}$ does not vanish identically at
any point of $X$) in the first equality (with $\mathcal{A}$ any set
of real-valued continuous functions on $X$, which is either an
algebra separating the points or a well-separating class).
\end{corollary}

\begin{proof}
Since $X$ is completely regular, by Theorem 3 of \cite{com-JTP-05} a
large deviation principle implies condition $(iii)$ of Proposition
\ref{rate-funct}.   The first assertion  is then  a direct
consequence of Proposition \ref{rate-funct} and Lemma
\ref{lemma-continuity-compact-open} with
$\mathcal{T}=\mathcal{C}(X)$. The existence of $\Lambda(\cdot)$ on
${\mathcal{C}(X)}_{ba}$ follows from the generalized version of
Varadhan's theorem for bounded above continuous functions
(\cite{com-TAMS-03}, Theorem 3.3);  the last assertion follows then
from Theorem \ref{Stone-W-exp-tight}.
\end{proof}

The following Prohorov-type result generalizes to any topological
space the preceding versions for normal Hausdorff and completely
regular spaces   given in \cite{com-TAMS-03} and \cite{com-JTP-05},
respectively. The notations $\Lambda^{(\mu_\beta^{t_\beta})}(h)$
means that the limit is taken along the subnet
$(\mu_\beta^{t_\beta})$ (i.e.
$\Lambda^{(\mu_\beta^{t_\beta})}(h)=\log\lim\mu_{\beta}^{t_{\beta}}(e^{h/t_{\beta}})$).

\begin{corollary}\label{general-Prohorov-LDP}
If
 $(\mu_\alpha)$ is $X_0$-exponentially tight with respect to
$(t_\alpha)$ and $\mathcal{C}(X)$ separates the points, then
 $(\mu_\alpha)$ has a subnet $(\mu_\beta)$ satisfying the
 conclusions of Theorem \ref{Stone-W-exp-tight} with any set
  $\mathcal{A}$ of real-valued continuous functions on $X$, which
is either an algebra separating the points or a well-separating
class.
\end{corollary}

\begin{proof}
Let $\mathcal{A}\subset\mathcal{C}(X)$ as above and
 put $\Lambda_\alpha(h)=\log\mu_\alpha^{t_\alpha}(e^h)$ for all
$h\in \mathcal{A}$, so that $(\Lambda_\alpha)$ is a net in the
compact set $[-\infty,+\infty]^{\mathcal{A}}$ (provided with the
product topology). Therefore,    $(\Lambda_\alpha)$ has a converging
subnet $(\Lambda_\beta)$, i.e.,
$\lim\Lambda_\beta(h)=\Lambda^{(\mu_\beta^{t_\beta})}(h)$ exists for
all $h\in\mathcal{A}$. The conclusion follows from Theorem
\ref{Stone-W-exp-tight}  applied to  $(\nu_\beta^{t_\beta})$.
\end{proof}

Corollary \ref{appli-tvs} generalizes Theorem 7.1 of
\cite{bry-Ann-Prob-92} where $X$ is required to be metrizable, and
also strengthens it by only requiring  the existence of
$\Lambda(\cdot)$  on $\mathcal{A}_-$ and giving the rate function in
terms of $\mathcal{A}_-$  (by continuous affine functionals, we mean
those of the form $u+c$, where $u$ belongs to the topological dual
and $c$ is a real) .

\begin{corollary}\label{appli-tvs}
Let $X$ be a  locally convex real topological vector space and
assume that $(\mu_\alpha)$ is exponentially tight with respect to
$(t_\alpha)$. If $\Lambda(\cdot)$ exists on the negative part of the
set $\mathcal{A}$ of all finite minima of continuous affine
functionals, then $(\mu_\alpha)$ satisfies a large deviation
principle with powers $(t_\alpha)$ and rate function
\[J(x)=\sup_{h\in\mathcal{A}_{-}}\{h(x)-\Lambda(h)\}=\sup_{h\in\mathcal{A}}\{h(x)-
\overline{\Lambda}(h)\}=\sup_{h\in\mathcal{A}}\{h(x)-
\underline{\Lambda}(h)\}=\]
\[\sup_{h\in\mathcal{A}}\{h(x)-\lim_{M\rightarrow+\infty}
\underline{\Lambda}(h_M)\}\ \ \ \ \ \ \textit{for all $x\in X$}.\]
\end{corollary}

\begin{proof}
 Since the topological dual of a locally
convex real topological vector space separates the points,
$\mathcal{A}$ is a well-separating class and the result follows from
Theorem \ref{Stone-W-exp-tight}.
\end{proof}

\begin{example}
We briefly describe here  a regular Hausdorff space $X$ such that
$X\verb'\'X_0$ is a singleton; in other words, $X$ fails to be
completely regular because of one point (it is taken
 from Example 8 of \cite{ben}, and we refer to this paper for more
 details). In particular,  the usual version of Bryc's
 theorem does not work, but   $\mathcal{C}(X)$ separates the points so that
 Theorem \ref{Stone-W-exp-tight} and
  Corollary \ref{general-Prohorov-LDP} apply. Note that this is not a trivial property
  when complete regularity fails,
   since there are regular spaces (with more than one point)
      where $\mathcal{C}(X)$ is reduced to
 constants (\cite{dou}).

Let $H$ (resp. $K$) be the set of all irrational numbers belonging
(resp. not belonging) to the standard Cantor set, and let
$\psi:K\rightarrow H$ be a homeomorphism. Let $D=]0,1[^2\ \verb'\'\
\{(x,x):x\ \textnormal{rational}\}$ and define the following
topology on $D$. Each $(x,y)$ is isolated when
 $x\neq y$; a basic neighborhood of $(x,x)\in H^2$ (resp. $(x,x)\in K^2$) is of the form
 $[(x,x),(x+\varepsilon,x)[$ (resp. $[(x,x),(x,x-\varepsilon)[$) for some $\varepsilon>0$.
 For each
 integer $n$ let $D_n$ be a copy of $D$, and let $Y$ be the quotient
 space of the topological sum of the $D_n's$, where the equivalence
 relation is given by identifying each $(x,x)\in K_n^2$ with $(\psi(x),\psi(x))\in
 H_{n+1}^2$. Put $X=Y\cup\{p\}$ where $p$ is any extra point, and define
the   basic open neighborhoods of $p$ the sets of  the form
 \[
 G_{n,p}=\{p\}\cup (D_n\verb'\'\ \textnormal{diagonal of $D_n$})\cup\bigcup_{m>n} D_m
 \]
 for some integer $n$.
 It is easy to verify that $X$ is  regular Hausdorff  with $X\verb'\'\{p\}\subset
 X_0$; in fact, every point of $X\verb'\'\{p\}$  has a clopen neighborhood basis. However,
  $p$ cannot be separated from  $X\verb'\'G_{n,p}$
  (for
any $G_{n,p}$) by
 a continuous function, i.e., $X\verb'\'{X_0}=\{p\}$ and $X$ is not
 completely regular.
 \end{example}


\begin{remark}\label{rem-algebra}
In the algebraic case, when $\mathcal{A}\subset\mathcal{C}_b(X)$ and
$\Lambda(\cdot)$ exists on $\mathcal{A}$, then $\Lambda(\cdot)$
exists on the unital algebra
 $\mathcal{B}=\mathcal{A}+\mathbb{R}$, and  on its uniform
 closure $\mathcal{B}'$, which is stable by finite minima, as it is well known
 (\cite{ped}, Lemma 4.3.3).
 Since $\mathcal{B}'$
 is then a well-separating class, in the   completely regular Hausdorff case,
 it follows from Bryc's theorem that under exponential
 tightness, large deviations hold with rate function
 \begin{equation}\label{rem-algebra-eq1}
 \forall x\in X,\ \ \ \ \
 J(x)=\sup_{h\in\mathcal{C}_b(X)}\{h(x)-\Lambda(h)\}.
 \end{equation}
 If moreover $X$ is metric,
  then the $\sup$ in (\ref{rem-algebra-eq1}) can be taken on $\mathcal{B}'$ by \cite{eic},
  and then on $\mathcal{B}$  by the continuity of $\Lambda(\cdot)$ with respect to the uniform
 metric,  hence finally on $\mathcal{A}$. This shows that when $X$
 is metric and
  $\mathcal{A}\subset\mathcal{C}_b(X)$ the first conclusion  of
  Theorem \ref{Stone-W-exp-tight}  follows
  easily from known results; however,  the expression of the rate
  function in terms of $\mathcal{A}_-$ seems new, except
   when
  $\mathcal{A}=\mathcal{C}_b(X)$ and $X$ Polish (see the proof of
  Bryc's theorem given in \cite{dup}). Also in the Polish case,
  the fact that
   the existence of $\Lambda(\cdot)$ on
${\mathcal{C}_b(X)}_-$ implies large deviations
 was already known (see the proof of Corollary 1.2.5 of
\cite{dup}).
\end{remark}

\begin{remark}\label{strengthening-well-sep}
In the completely regular Hausdorff case,  the strengthening of the
known  versions with well-separating class consists in the form of
the rate function in terms of $\overline{\Lambda}(\cdot)$ (resp.
$\lim_{M\rightarrow+\infty}\underline{\Lambda}(\cdot_M)$) on
$\mathcal{A}$, and in terms of $\Lambda(\cdot)$ on $\mathcal{A}_-$;
the formula with $\underline{\Lambda}(\cdot)$ was known (see
Exercise 4.4.1 of \cite{dem}). When $\Lambda(\cdot)$ exists on the
whole well-separating class, (\ref{Stone-W-exp-tight-eq1}) was
proved in \cite{eic} for a sequence $(\mu_n^{1/n})$ and  $X$ metric
(the result is stated there for normal spaces but the author uses a
metric in the proof). Unlike \cite{eic}, Proposition
\ref{rate-funct} allows to prove (\ref{Stone-W-exp-tight-eq1})
without using the stability by finite minima; nevertheless, this
property is required to get the existence of $\Lambda(\cdot)$.
\end{remark}

\section{The case locally compact}\label{section-locally-compact}

In this section we assume that $X$ is  locally compact regular,  we
drop the exponential tightness hypothesis, and we
  consider an algebra $\mathcal{A}$  of real-valued continuous functions on $X$
  vanishing at
$\infty$, which separates the points and does not vanish identically
at any point of $X$.

 We first  recall some  basic topological facts. Here, the notion of local compactness is the one of
  \cite{kel}, that is,  each point has a compact neighborhood.
  This definition
  differs from others ones where it is asked that each point has an
  open neighborhood with compact closure (see for example
  \cite{eng}).
    Such a space need not be regular; however, it is known that
    $X$ is regular if and only if $X$ is completely regular, and
    this holds in particular when $X$ is Hausdorff. The  one-point compactification
     $\hat{X}$ of $X$ (where the neighborhood of $\infty$ are given by the complements of
     closed compact subsets of $X$) is
Hausdorff (resp. regular) if and only if $X$ is Hausdorff (resp.
regular) (\cite{kel}).

Let us look the following example: Take $X=\mathbb{N}\verb'\'\{0\}$
and $\mu_n$ the uniform probability measure on $\{1,...,n\}$ for all
$n\in\mathbb{N}\verb'\'\{0\}$;
 then, large deviations
holds with rate function $J=0$, so that $\Lambda(h)=\sup_{x\in X}
h(x)\ge 0$ for all $h\in\mathcal{A}$, which implies
\[\sup_{h\in
\mathcal{A}}\{-\Lambda(h)\}\ge\sup_{x\in
X,h\in\mathcal{A}}\{h(x)-\Lambda(h)\}.\]

 The next theorem  shows that the above inequality is
 in general   sufficient (and necessary in the Hausdorff case) to get large
deviations with a rate function $J$ satisfying a condition weaker
than the tightness (namely (\ref{Stone-Wei-no-tight-eq0})), and having still the form
(\ref{Stone-Wei-no-tight-eq2}); indeed,  $J$ is tight if and only if
the L.H.S. of (\ref{Stone-Wei-no-tight-eq0}) equals $+\infty$.
Theorem \ref{Stone-Wei-no-tight} supplies a broad class of examples
where the classical version of Bryc's theorem does not work, since
any net of Borel probability measures on any regular locally compact
space satisfying a large deviation principle with no-tight rate
function, will satisfy the  hypotheses.

\begin{theorem}\label{Stone-Wei-no-tight}
 Let  $X$ be locally compact regular,  let
  $\mathcal{A}$ be an algebra
of real-valued continuous functions on $X$
  vanishing at
$\infty$, which separates the points and does not vanish identically
at any point of $X$, and assume that
 $\Lambda(\cdot)$ exists on $\mathcal{A}$. The following conclusions
 hold.
\begin{itemize}
\item[(a)]  $(\mu_{\alpha})$ satisfies a vague large deviation principle with powers
$(t_{\alpha})$ and rate function
\begin{equation}\label{Stone-Wei-no-tight-eq2}
J(x)=\sup_{h\in\mathcal{A}} \{h(x)-\Lambda(h)\}\ \ \ \ \ \
\textnormal{for all $x\in X$}.
\end{equation}
 If moreover
 \begin{equation}\label{Stone-Wei-no-tight-eq1}
\sup_{h\in\mathcal{A}}\{-\Lambda(h)\}\ge \sup_{x\in
X,h\in\mathcal{A}}\{h(x)-\Lambda(h)\},
\end{equation} then the  large deviation principle holds with the same rate
function, and
\begin{equation}\label{Stone-Wei-no-tight-eq0}
\sup_{K\in\mathcal{K}}\inf_{x\in X\verb'\'K}J(x)\ge\sup_{x\in
X}J(x).
\end{equation}
\item[(b)] If  $X$ is Hausdorff and $(\mu_{\alpha})$ satisfies a large deviation
principle with powers $(t_{\alpha})$, then the rate function is
given by (\ref{Stone-Wei-no-tight-eq2}), and
 (\ref{Stone-Wei-no-tight-eq0}) implies
(\ref{Stone-Wei-no-tight-eq1}).
\item[(c)] $(\mu_{\alpha})$ is
 exponentially tight with respect to $(t_{\alpha})$ if and only if
 \[\sup_{h\in\mathcal{A}}\{-\Lambda(h)\}=+\infty.\]
\end{itemize}
\end{theorem}

\begin{proof}
 Let $\phi:X\rightarrow\hat{X}$ be the canonical imbedding of $X$
in its one-point compactification $\hat{X}=X\cup\{\infty\}$, let
$\mathcal{C}_0(X)$ be the algebra of  continuous functions on $X$
  vanishing at
$\infty$,  and identify $\mathcal{C}_0(X)$ with the set of
continuous functions $h$ on $\hat{X}$ such that $h(\infty)=0$. The
hypotheses imply that the algebra $\mathcal{A}+\mathbb{R}$ separates
the points of $\hat{X}$. By the Stone-Weierstrass theorem in
$\hat{X}$, for each $g\in\mathcal{C}_0(X)$ and each $\varepsilon>0$,
there exists $h\in\mathcal{A}$ and $c\in\mathbb{R}$ such that
$\sup_{x\in \hat{X}}|g(x)-(h(x)+c)|\le\varepsilon$, hence
$\sup_{x\in X}|g(x)-h(x)|\le 2\varepsilon$ since
$g(\infty)=h(\infty)=0$. It follows that $\mathcal{A}$ is uniformly
dense in $\mathcal{C}_0(X)$, hence $\Lambda(\cdot)$ exists on
$\mathcal{C}_0(X)$.
 For each $x\in X$
and each $G\in\mathcal{G}$ containing $x$,
 there exists
$h_{G,x}\in\mathcal{C}_0(X)$ satisfying $1_{\{x\}}\le h_{G,x}\le
1_G$, so that $1_{\{x\}}\le e^{sh_{G,x}-s}\le 1_G\vee e^{-s}$ for
all $s>0$; in particular, when $G=X\verb'\'K$ for some
$K\in\mathcal{K}$, we have
 $1_{\{x\}}\le
e^{sh_{G,x}-s}\le 1_{\hat{X}\verb'\'K}\vee e^{-s}$. For each
$K\in\mathcal{K}$, there is a continuous function $h_{K,\infty}$ on
$\hat{X}$ such that $1_{\{\infty\}}\le h_{K,\infty}\le
1_{\hat{X}\verb'\'K}$, so that $1_{\{\infty\}}\le
e^{s(h_{K,\infty}-1)}\le 1_{\hat{X}\verb'\'K}\vee e^{-s}$ for all
$s>0$, with $h_{K,\infty}-1\in\mathcal{C}_0(X)$. Thus, the algebra
$\mathcal{B}=\mathcal{C}_0(X)+\mathbb{R}$ is an approximating class
for $\hat{X}$. Since
$\lim\phi[\mu_{\alpha}]^{t_{\alpha}}(e^{h/t_{\alpha}})$
 exists for all
$h\in\mathcal{B}$, it follows from Corollary 3 of \cite{com-JTP-05}
that $(\phi[\mu_{\alpha}])$ satisfies a large deviation principle
with powers $(t_{\alpha})$. Since $\mathcal{B}$ satisfies the
hypotheses of Corollary 2 of \cite{com-JTP-05}, the rate function
$\hat{J}$ is given by
\begin{equation}\label{Stone-Wei-no-tight-eq3}
\hat{J}(x)=\sup_{\{h\in\mathcal{B}:
h(x)=0\}}\{-\Lambda(h)\}=\sup_{h\in\mathcal{B}}\{h(x)-\Lambda(h)\}=
\sup_{h\in\mathcal{C}_0(X)}\{h(x)-\Lambda(h)\}
\end{equation}
\[=\sup_{h\in\mathcal{A}}
\{h(x)-\Lambda(h)\}\ \ \ \ \ \ \textnormal{for all $x\in \hat{X}$},
\]
where the last equality follows by noting that
 for each $h\in\mathcal{C}_0(X)$,  $\Lambda(h)=\lim\Lambda(h_i)$ for all
nets $(h_i)$ in $\mathcal{A}$ converging uniformly to $h$. In
particular
\begin{equation}\label{Stone-Wei-no-tight-eq3.1}
\hat{J}(\infty)=\sup_{h\in\mathcal{A}}\{-\Lambda(h)\}.
 \end{equation}
 The complete regularity of  $\hat{X}$ yields
\[
e^{-\hat{J}(\infty)}=
\inf_{K\in\mathcal{K}}\limsup\phi[\mu_{\alpha}]^{t_{\alpha}}
(\hat{X}\verb'\'K)\\=\inf_{K\in\mathcal{K}}\limsup\mu_{\alpha}^{t_{\alpha}}
(X\verb'\'K),
\]
which proves (c). Then, we have
\begin{equation}\label{Stone-Wei-no-tight-eq3.2}
\liminf\mu_{\alpha}^{t_{\alpha}}(G)=\liminf\phi[\mu_{\alpha}]^{t_{\alpha}}(G)\ge
\sup_{x\in G}e^{-\hat{J}(x)}\ \ \ \ \ \textnormal{for all
$G\in\mathcal{G}$},
\end{equation} and
\begin{equation}\label{Stone-Wei-no-tight-eq3.3}
\limsup\mu_{\alpha}^{t_{\alpha}}(K)=\limsup\phi[\mu_{\alpha}]^{t_{\alpha}}(K)
\le\sup_{x\in K}e^{-\hat{J}(x)}\ \ \ \ \ \ \textnormal{for all
$K\in\mathcal{K}$},
\end{equation}
and so, $(\mu_{\alpha})$ satisfies  a vague large deviation
principle with power $(t_\alpha)$ and  rate function
(\ref{Stone-Wei-no-tight-eq2}), since $J=\hat{J}_{\mid
 X}$. This proves the first assertion of (a).

 Assume moreover  that (\ref{Stone-Wei-no-tight-eq1}) holds.  By (\ref{Stone-Wei-no-tight-eq3}) and
 (\ref{Stone-Wei-no-tight-eq3.1}),
(\ref{Stone-Wei-no-tight-eq1}) is equivalent to
\begin{equation}\label{Stone-Wei-no-tight-eq4}
e^{-\hat{J}(\infty)}\le\inf_{x\in X} e^{-\hat{J}(x)},
\end{equation}
 hence for
each  $F\in\mathcal{F}$,
\[
\limsup\mu_{\alpha}^{t_{\alpha}}(F)=\limsup\phi[\mu_{\alpha}]^{t_{\alpha}}(F\cup\{\infty\})
\le\sup_{x\in F\cup\{\infty\}}e^{-\hat{J}(x)}\le\sup_{x\in
F}e^{-\hat{J}(x)},
\]which proves the large deviation
upper bounds with  $J$. The lower semi-continuity of $\hat{J}$ at
$\infty$, and  (\ref{Stone-Wei-no-tight-eq4}) imply
\[\sup_{x\in X}J(x)\le\hat{J}(\infty)=\sup_{K\in\mathcal{K}}\inf_{x\in
\hat{X}\verb'\'K}\hat{J}(x)=\sup_{K\in\mathcal{K}}\inf_{x\in
X\verb'\'K}J(x),\] and (\ref{Stone-Wei-no-tight-eq0}) holds. The
second assertion of (a) is proved.

Assume that $X$ is Hausdorff and   $(\mu_{\alpha})$ satisfies a
large deviation principle with powers $(t_\alpha)$.
 Then,
$\Lambda(\cdot)$ exists on $\mathcal{C}_0(X)$, and (as proved
before) $(\phi[\mu_{\alpha}])$ satisfies a large deviation principle
with powers $(t_{\alpha})$ and rate function $\hat{J}$
 defined in (\ref{Stone-Wei-no-tight-eq3}). Therefore,
 $(\mu_{\alpha})$ satisfies a vague large deviation principle
with power $(t_\alpha)$ and rate function $\hat{J}_{\mid X}$, and by
uniqueness of a vague rate function on a locally compact Hausdorff
space, the rate function coincides with $\hat{J}_{\mid X}$. This
proves the first conclusion of (b).  Assume  that
(\ref{Stone-Wei-no-tight-eq0}) holds, and define the function
$\hat{l}$ on $\hat{X}$ by
\[\hat{l}(x)=\left\{
\begin{array}{ll}
J(x) & \ \ \ \ \ \ \textnormal{if $x\neq\infty$}
\\
\sup_{y\in X}J(y)& \ \ \ \ \ \ \textnormal{if $x=\infty$}.
\end{array}
\right.
\]
If $\hat{l}(\infty)>\sup_{K\in\mathcal{K}}\inf_{x\in
\hat{X}\verb'\'K}\hat{l}(x)$, then
\[\hat{l}(\infty)>\sup_{K\in\mathcal{K}}\inf_{x\in
X\verb'\'K}\hat{l}(x)=\sup_{K\in\mathcal{K}}\inf_{x\in
X\verb'\'K}J(x),
\] which  contradicts
(\ref{Stone-Wei-no-tight-eq0}). Therefore, $\hat{l}$ is lower
semi-continuous at $\infty$,  and so  $\hat{l}$ is lower
semi-continuous on $\hat{X}$. The large deviations for
$(\mu_{\alpha}^{t_{\alpha}})$ with rate function $J$ implies for
each $F\in\mathcal{F}$ and each $K\in\mathcal{K}$ with $F\subset
X\verb'\'K$,
\[
\limsup\mu_{\alpha}^{t_{\alpha}}(F)=\limsup\phi[\mu_{\alpha}]^{t_{\alpha}}(F\cup\{\infty\})
\le\sup_{x\in F}e^{-J(x)}=\sup_{x\in
F\cup\{\infty\}}e^{-\hat{l}(x)}\]\[\le\sup_{x\in
X\verb'\'K}e^{-J(x)}\le\sup_{x\in
\hat{X}\verb'\'K}e^{-\hat{l}(x)}\le\liminf\mu_{\alpha}^{t_{\alpha}}(X\verb'\'K)=
\liminf\phi[\mu_{\alpha}]^{t_{\alpha}}(\hat{X}\verb'\'K).
\]
Together with (\ref{Stone-Wei-no-tight-eq3.2}) and
(\ref{Stone-Wei-no-tight-eq3.3}), this shows that
$(\phi[\mu_{\alpha}])$ satisfies a large deviation principle with
powers $(t_{\alpha})$ and rate function $\hat{l}$, hence
$\hat{l}=\hat{J}$ and so $\hat{J}(\infty)=\sup_{x\in X}J(x)$, which
implies  (\ref{Stone-Wei-no-tight-eq1}) by
(\ref{Stone-Wei-no-tight-eq3}) and (\ref{Stone-Wei-no-tight-eq3.1}).
\end{proof}

The two following corollaries are the analogues of Corollaries
 \ref{rate-funct-Polish} and \ref{general-Prohorov-LDP} respectively,
for locally compact spaces when the tightness  fails (exponential or
of the rate function).

\begin{corollary}\label{loc-compact-rate-funct}
 Let  $X$ be  locally compact Hausdorff. If
   $(\mu_\alpha)$ satisfies a large deviation principle with powers
$(t_\alpha)$,  then  the  rate function $J$ is given by
\[
J(x)=\sup_{h\in\mathcal{A}}\{h(x)-\Lambda(h)\}\ \ \ \ \ \
\textit{for all $x\in X$},
\]
where $\mathcal{A}$ is any algebra of real-valued continuous
functions on $X$ vanishing at $\infty$, which separates the points
and does not vanish identically at any point of $X$.
\end{corollary}

\begin{proof}
The large deviation principle implying  the existence of
$\Lambda(\cdot)$ on $\mathcal{A}$ by the tightness-free version of
Varadhan's   theorem (\cite{com-TAMS-03}), the conclusion follows
from Theorem \ref{Stone-Wei-no-tight} (b).
\end{proof}

\begin{corollary}\label{local-compact-tight-free-Prohorov}
Let  $X$ be  locally compact regular. Let $\mathcal{A}$ be an
algebra of real-valued continuous functions on $X$ vanishing at
$\infty$, which separates the points and does not vanish identically
at any point of $X$. If
\[
\sup_{h\in\mathcal{A}}\{-\overline{\Lambda}(h)\}\ge \sup_{x\in
X,h\in\mathcal{A}}\{h(x)-\underline{\Lambda}(h)\},
\]
then
 $(\mu_\alpha)$ has a subnet $(\mu_\beta)$ satisfying a large
deviation principle with powers $(t_\beta)$ and  rate function
\[
J(x)=\sup_{h\in\mathcal{A}}
\{h(x)-\Lambda^{(\nu_\beta^{t_\beta})}(h)\}\ \ \ \ \ \
\textnormal{for all $x\in X$};\]  moreover we have
\[\sup_{K\in\mathcal{K}}\inf_{x\in X\verb'\'K}J(x)\ge\sup_{x\in
X}J(x).\]
\end{corollary}

\begin{proof}
The same argument as for Corollary \ref{general-Prohorov-LDP} gives
the existence of $\Lambda^{(\nu_\beta^{t_\beta})}(\cdot)$ on
$\mathcal{A}$ for some subnet $(\nu_\beta^{t_\beta})$. The
conclusion follows from Theorem  \ref{Stone-Wei-no-tight} (a) since
the hypothesis implies (\ref{Stone-Wei-no-tight-eq1}) with
$\Lambda^{(\nu_\beta^{t_\beta})}(h)$ in place of $\Lambda(h)$.
\end{proof}

\begin{remark}
The assumption of regularity in Theorem \ref{Stone-Wei-no-tight} is
necessary to have  the complete regularity of the one-point
compactification; if one drop it, then we fall in the general case
of the preceding section. The Hausdorff assumption in (b) is only
required  to ensure the uniqueness of the rate function for  vague
large deviations; consequently, the same conclusions hold for any
regular locally compact space satisfying this property. The
assumption that $\mathcal{A}$ does not vanish identically at any
point of $X$ is crucial in the proof: it ensures that
$\mathcal{A}+\mathbb{R}$ separates the points of $\hat{X}$, which
allows to prove the first assertion of (a), part on which is built
the rest of the proof.
\end{remark}

\begin{remark}
Under the hypotheses of Theorem \ref{Stone-Wei-no-tight} and when
$X$ is not compact, $(\ref{Stone-Wei-no-tight-eq1})$ is equivalent
to
\begin{equation}\label{LDP-Alex-compactif-eq6}
\sup_{h\in\mathcal{A}}\{-\Lambda(h)\}=\sup_{x\in
X,h\in\mathcal{A}}\{h(x)-\Lambda(h)\}.
\end{equation}
In particular the exponential tightness holds if and only if
\[\sup_{x\in X} J(x)=+\infty.\] Indeed, assume that
$(\ref{Stone-Wei-no-tight-eq1})$ holds. The set
$\{\hat{J}>\sup_{x\in X} \hat{J}(x)\}$ is open in $\hat{X}$ and
since $\{\infty\}$ is not open when  $X$ is not compact, we have
$\hat{J}(\infty)\le\sup_{x\in X} \hat{J}(x)$; since the converse
inequality holds by  $(\ref{Stone-Wei-no-tight-eq1})$, we have
$\hat{J}(\infty)=\sup_{x\in X}\hat{J}(x)$, which is equivalent to
(\ref{LDP-Alex-compactif-eq6}).
\end{remark}

\section*{Acknowledgments}
 This work
has been supported by DICYT-USACH grant No. 040533C.

\bibliographystyle{amsplain}

\end{document}